\documentclass[a4paper,11pt]{article}

\usepackage[top=2.5cm, bottom=3.5cm, left=2.5cm, right=2.5cm]{geometry}
\usepackage[utf8]{inputenc}
\usepackage{amsmath}
\usepackage{amsfonts}
\usepackage{amssymb}
\usepackage{mathrsfs}               
\usepackage{amsthm}
\usepackage{color}
\usepackage{stmaryrd}		    

\usepackage[hidelinks]{hyperref}

\newtheorem{theorem}{Theorem}[section]
\newtheorem{lemma}[theorem]{Lemma}
\newtheorem{corollary}[theorem]{Corollary}
\newtheorem{remark}[theorem]{Remark}
\newtheorem{proposition}[theorem]{Proposition}

\theoremstyle{definition}
\newtheorem{definition}[theorem]{Definition}

\newcommand{\sgn}[0]{\,\mathrm{sgn}}
\newcommand{\dd}{\,\mathrm{d}}
\newcommand{\DD}{\mathrm{D}}
\newcommand{\R}{\mathbb{R}}
\newcommand{\F}{\mathcal{F}}
\newcommand{\N}{\mathbb{N}}
\newcommand{\Z}{\mathbb{Z}}
\renewcommand{\L}{\mathcal{L}}

\newcommand{\1}{\mathbf{1}}
\renewcommand{\P}{\mathbb{P}}

\newcommand{\supp}{\mathrm{supp}}

\newcommand{\var}{\mathrm{var}}
\newcommand{\V}{\mathcal{V}}
\newcommand{\B}{\mathcal{B}}

\mathchardef\mhyphen="2D

\title{Local times for typical price paths and\\ pathwise Tanaka formulas 
\thanks{We are grateful to Peter Imkeller and Johannes Ruf for helpful discussions on the subject matter. We would like to thank Alexander Cox for pointing out a small mistake in a previous version of the paper.}}

\author{Nicolas Perkowski\thanks{N.P. is supported by the Fondation Sciences Math\'ematiques de Paris  (FSMP) and by a public grant overseen by the French National Research Agency (ANR) as part of the ``Investissements d'Avenir'' program (reference: ANR-10-LABX-0098).} \\ 
CEREMADE \& CNRS UMR 7534 \\ 
Universit\'e Paris-Dauphine \\
\texttt{perkowski@ceremade.dauphine.fr}
\and David J. Pr\"omel\thanks{D.J.P. is supported by a Ph.D. scholarship of the DFG Research Training Group 1845 "Stochastic Analysis with Applications in Biology, Finance and Physics".} \\
Humboldt-Universit\"at zu Berlin \\
Institut f\"ur Mathematik \\
\texttt{proemel@math.hu-berlin.de}}

\begin{document}

\maketitle

\begin{abstract}
  Following a hedging based approach to model free financial mathematics, we prove that it should be possible to make an arbitrarily large profit by investing in those one-dimensional paths which do not possess local times. The local time is constructed from discrete approximations, and it is shown that it is $\alpha$-H\"older continuous for all $\alpha<1/2$. Additionally, we provide various generalizations of F\"ollmer's pathwise It\^o formula.
\end{abstract}


%

\section{Introduction}

This paper uses Vovk's~\cite{Vovk2012} game-theoretic approach to mathematical finance to construct local times for ``typical price paths''.  Vovk's approach is based on an outer measure, which is given by the cheapest pathwise superhedging price, and it does not presume any probabilistic structure. We define discrete versions of the local time and prove that outside a set of outer measure zero they converge to a continuous limit. Roughly speaking, this means that it should be possible to make an arbitrarily large profit by investing in those paths where the convergence of the discrete local times fails. A nice consequence is that the convergence takes place quasi surely under all semimartingale measures for which the coordinate process satisfies the classical condition of "no arbitrage opportunities of the first kind", i.e. for which the drift has a square integrable density with respect to the quadratic variation of the local martingale part.

Using these pathwise local times, we derive various pathwise change of variable formulas which generalize F\"ollmer's pathwise It\^o formula~\cite{Follmer1981} in the same way that the classical Tanaka formula generalizes the classical It\^o formula. In particular, we can integrate $f(S)$ against a typical price path $S$ whenever $f$ has finite $q$-variation for some $q<2$.

This work is a continuation of~\cite{Perkowski2013}, where we used Vovk's approach to show that in a multidimensional setting every typical price path has a natural It\^o rough path in the sense of Lyons~\cite{Lyons1998} associated to it. Based on this, we set up a pathwise theory of integration which was motivated by possible applications in model free financial mathematics. 
With the techniques of~\cite{Perkowski2013} we are able to treat integrands that are not necessarily functions of the integrator. But if we want to construct $\int f(S) \dd S$, then we need $f \in C^{1+\varepsilon}$. The aim of the current paper is to show that for one-dimensional price processes this assumption can be greatly relaxed.

Our motivation comes amongst others from~\cite{Davis2014}, where pathwise local times and a pathwise generalized It\^o formula are used to derive arbitrage free prices for weighted variance swaps in a model free setting. The techniques of~\cite{Davis2014} allow to handle integrands in the Sobolev space $H^1$. Here we extend this to not necessarily continuous integrands of finite $q$-variation for some $q<2$. Further motivations can be found in the survey paper~\cite{Follmer2013} which emphasizes possible applications of pathwise integration to robust hedging problems, or in~\cite{Carr1990} and \cite{Sondermann2006}, where local times appear naturally in a financial context and are used to resolve the so-called ``stop-loss start-gain paradox''.

We refer to~\cite{Perkowski2013} for a more detailed discussion of the need for pathwise stochastic integrals in model free finance.

\subsubsection*{Plan of the paper}

In Section~\ref{sec:tanaka} we present various extensions of F\"ollmer's pathwise It\^o formula under suitable assumptions on the local time. In Section~\ref{sec:localtime} we show that typical price paths possess local times which satisfy all the assumptions of Section~\ref{sec:tanaka}.

\section{Pathwise Tanaka formula}\label{sec:tanaka}

A first non-probabilistic approach to stochastic calculus was introduced by F\"ollmer in \cite{Follmer1981}, where an It\^o formula was developed for a class of real-valued functions 
with quadratic variation. This builds our starting point for a pathwise version of Tanaka's formula and a generalized  It\^o formula, respectively. Let us start by recalling F\"ollmer's definition of quadratic variation. \\
A \textit{partition} $\pi$ is an increasing sequence $0 = t_0 < t_1 < \ldots$ without accumulation points, possibly taking the value $\infty$. For $T>0$ we denote by $\pi[0,T]:= \{ t_j \, : \, t_j \in [0,T) \}\cup \{T\}$ the partition $\pi$ restricted to $[0,T]$, and if $S\colon [0,\infty)\to \R$ is a continuous function we write
\begin{equation*}
   m(S,\pi[0,T]):=\max_{ t_j \in \pi[0,T]\setminus \{t_0 \}} \vert S(t_j) - S(t_{j-1})\vert
\end{equation*}
for the mesh size of $\pi$ along $S$ on the interval $[0,T]$. We denote by $\mathcal{B}([0,\infty))$ the Borel $\sigma$-algebra on $[0,\infty)$. 

\begin{definition}\label{def:qudratic}
  Let $(\pi^n)$ be a sequence of partitions and let $S \in C([0,\infty), \R)$ be such that $\lim_{n \to \infty} m(S,\pi^n[0,T]) =0$ for all $T > 0$. We say that $S$ has \textit{quadratic variation} along $(\pi^n)$ if the sequence of measures
  \begin{equation*}
    \mu_n := \sum_{t_j \in \pi^n \setminus  \{\infty \}} ( S(t_{j+1}) - S(t_j))^2 \delta_{t_j}, \qquad n \in \N,
  \end{equation*}
  on $([0,\infty), \B([0,\infty)))$ converges vaguely to a nonnegative Radon measure $\mu$ without atoms, where $\delta_t$ denotes the Dirac measure at $t \in [0, \infty)$. We write $\langle S \rangle(t) := \mu([0,t])$ for the continuous ``distribution function'' of $\mu$ and $\mathcal{Q}(\pi^n)$ for the set of all continuous functions having quadratic variation along $(\pi^n)$. 
\end{definition}

The reason for only requiring $\lim_n m(S,\pi^n[0,T]) =0$ rather than assuming that the mesh size of $(\pi^n)$ goes to zero is that later we will work with Lebesgue partitions and paths with piecewise constant parts, in which case only the first assumption holds.

We stress the fact that $\mathcal{Q}(\pi^n)$ depends on the sequence $(\pi^n)$ and that for a given path the quadratic variation along two different sequences of partitions can be different, even if both exist. This is very unpleasant and might lead the reader to question the usefulness of our results. But quite remarkably there is a large class of paths which have a natural pathwise quadratic variation that is independent of the specific partition used to calculate it. More precisely, in the master's thesis~\cite{Lemieux1983}, see also~\cite{Chacon1981}, the notion of \emph{quadratic arc length} is introduced. Roughly speaking, a path $S$ has quadratic arc length $A$ if the quadratic variation of $S$ along \emph{any sequence of Lebesgue partitions} is equal to $A$. It is shown in~\cite{Lemieux1983}, Theorem III.3.3, that almost every sample path $S(\omega)$ of a continuous semimartingale has a quadratic arc length which is equal to the semimartingale quadratic variation $\langle S \rangle(\omega)$. The same theorem also shows that almost every sample path of a continuous semimartingale has a natural local time which can be obtained by counting interval upcrossings.

For $k \in \N$ let us write $C^k = C^k(\R,\R)$ for the space of $k$ times continuously differentiable functions, and $C^k_b = C^k_b(\R,\R)$ for the space of functions in $C^k$ that are bounded with bounded derivatives, equipped with the usual norm $\|\cdot\|_{C^k_b}$.

\begin{theorem}[\cite{Follmer1981}]\label{thm:ito}
  Let $(\pi^n)$ be a sequence of partitions and let $S \in \mathcal{Q}(\pi^n)$ and $f \in C^2$. Then the pathwise It\^o formula
  \begin{equation*}
    f(S(t)) = f(S(0)) + \int_0^t f^\prime(S(s)) \dd S(s) + \frac{1}{2} \int_0^t f^{\prime \prime} (S(s)) \dd \langle S \rangle(s)
  \end{equation*}
  holds with 
  \begin{equation}\label{eq:int}
    \int_0^t f^\prime (S(s)) \dd S(s) := \lim_{n \to \infty} \sum_{t_j \in \pi^n} f^\prime (S(t_j)) (S(t_{j+1}\wedge t) - S(t_j \wedge t)), \quad t \in [0,\infty),
  \end{equation}
  where the series in \eqref{eq:int} is absolutely convergent. 
  
  In particular, the integral $\int_0^\cdot g(S(s)) \dd S(s)$ is defined for all $g \in C^1$, and for all $T>0$ the map $C^1_b \ni g \mapsto \int_0^\cdot g(S(s)) \dd S(s) \in C([0,T],\R)$ defines a bounded linear operator and we have
  \begin{equation*}
    \Big| \int_0^t g(S(s))\dd S(s) \Big| \le |S(t) - S(0)| \times \|g\|_{L^\infty(\supp(S|_{[0,t]}))} + \frac{1}{2} \langle S\rangle(t) \| g'\|_{L^\infty(\supp(S|_{[0,t]}))}
  \end{equation*}
  for all $t \ge 0$, where $\supp(S|_{[0,t]}])$ denotes the support of $S$ restricted to the interval $[0,t]$.
\end{theorem}

F\"ollmer actually requires the mesh size $\max_{ t_j \in \pi^n\setminus \{t_0\},\, t_j \le T} \vert t_j - t_{j-1}\vert$ to converge to zero for all $T >0$, but he also considers c\`adl\`ag functions $S$. For continuous $S$, the proof only uses that $m(S,\pi^n[0,T])$ converges to zero.

The continuity of the It\^o integral is among its most important properties: if we approximate the integrand in a suitable topology, then the approximate integrals converge in probability to the correct limit. This is absolutely crucial in applications, for example when solving stochastic optimization problems or SDEs. Here we are arguing for one fixed path, so the statement in Theorem~\ref{thm:ito} is a natural formulation of the continuity properties in our context.

In the theory of continuous semimartingales, It\^o's formula can be extended further to a generalized It\^o rule for convex functions, see for instance Theorem 6.22 in \cite{Karatzas1988}. In the spirit of F\"ollmer, a generalized It\^o rule for functions in suitable Sobolev spaces was derived in the unpublished diploma thesis of Wuermli~\cite{Wuermli1980}. We briefly recall here the idea for this pathwise version as presented in~\cite{Wuermli1980} or~\cite{Davis2014}.

Let $f^\prime$ be right-continuous and of locally bounded variation, and we set $f(x) := \int_{(0,x]} f^\prime(y) \dd y$ for $x \ge 0$ and $f(x) := - \int_{(x,0]} f^\prime(y) \dd y$ for $x < 0$. Then we get for $b\ge a$ that
\[
   f(b) - f(a) = f^\prime(a)(b-a) + \int_{(a,b]} (f^\prime(x) - f^\prime(a)) \dd x = f^\prime(a)(b-a) + \int_{(a,b]} (b-t) \dd f^\prime(t),
\]
where we used integration by parts, and where the integral on the right hand side is to be understood in the Riemann-Stieltjes sense. For $b<a$, we get $f(b) - f(a) = f^\prime(a)(b-a) + \int_{(b,a]} (t-b) \dd f^\prime(t)$. Therefore, for any $S \in C([0,\infty], \mathbb{R})$ and any partition $\pi$ we have
\begin{align} \label{eq:discrete local time ito} \nonumber
  f(S( t)) - f(S(0)) &= \sum_{t_j \in \pi} f^\prime (S(t_j \wedge t )) (S(t_{j+1} \wedge t) - S(t_j \wedge t)) \\
         &\quad + \int_{-\infty}^\infty \Big( \sum_{t_j \in \pi}  \1_{\llparenthesis S(t_j\wedge t), S(t_{j+1}\wedge t) \rrbracket} (u) \vert S(t_{j+1}\wedge t) - u \vert \Big) \dd f^{\prime}(u),
\end{align}
where we used the notation 
\begin{align*}
  \llparenthesis u,v \rrbracket := \begin{cases}
                                     (u,v], & \text{if } u \leq v,\\
                                     (v,u], & \text{if } u > v, 
                                   \end{cases}
\end{align*}
for $u,v \in \mathbb{R}$.
Let us define a discrete local time by setting
\begin{equation*}
  L^{\pi}_t( S, u ) := \sum_{t_j \in \pi} \1_{\llparenthesis S(t_j \wedge t), S(t_{j+1} \wedge t) \rrbracket} (u) \vert S(t_{j+1} \wedge t) - u \vert, \quad u \in \mathbb{R},
\end{equation*}
and note that $L_t^{\pi}(S,u)= 0$ for $u \notin [\inf_{s\in[0,t]} S(s) ,  \sup_{s\in[0,t]} S(s) ]$. In the following we may omit the $S$ and just write $L^{\pi}_t(u)$.

\begin{definition}\label{def:local}
  Let $(\pi^n)$ be a sequence of partitions and let $S \in C([0,\infty),\mathbb{R})$. A function $L(S) \colon [0,\infty) \times \mathbb{R} \to \mathbb{R}$ is called $L^2$-\textit{local time} of $S$ along $(\pi^n)$ if for all $t \in [0,\infty)$ it holds $\lim_{n \to \infty} m(S,\pi^n[0,t]) =0$ and the discrete pathwise local times $L^{\pi^n}_t(S, \cdot)$ converge weakly in $L^2(\mathrm{d} u)$ to $L_t(S, \cdot)$ as $n \to \infty$. We write $\L_{L^2}(\pi^n)$ for the set of all continuous functions having an $L^2$-local time along $(\pi^n)$.
\end{definition}

Using this definition of the local time, Wuermli showed the following theorem, where we denote by $H^k = H^k(\R,\R)$ the Sobolev space of functions which are $k$ times weakly differentiable in $L^2(\R,\R)$.

\begin{theorem}[\cite{Wuermli1980}, Satz~9 or~\cite{Davis2014}, Proposition~B.4]\label{thm:odf}
  Let $(\pi^n)$ be a sequence of partitions and let $S \in \L_{L^2}(\pi^n)$. Then $S \in \mathcal{Q}(\pi^n)$, and for every $f \in H^2$ the generalized pathwise It\^o formula
  \begin{equation*}
    f(S(t)) = f(S(0)) + \int_0^t f^\prime(S(s)) \dd S(s) + \int_{-\infty}^\infty f^{\prime \prime} (u) L_t(S,u)\dd u
  \end{equation*}
  holds with 
  \begin{equation*}
    \int_0^t f^\prime (S(s)) \dd S(s) := \lim_{n \to \infty} \sum_{t_j \in \pi^n} f^\prime (S(t_j)) (S(t_{j+1} \wedge t) - S(t_j\wedge t)), \quad t \in [0,\infty).
  \end{equation*}
  (Note that $f^\prime$ is continuous for $f \in H^2$). In particular, the integral $\int_0^\cdot g(S(s)) \dd S(s)$ is defined for all $g \in H^1$, and for all $T>0$, the map $H^1 \ni g \mapsto \int_0^\cdot g(S(s)) \dd S(s) \in C([0,T],\R)$ defines a bounded linear operator.
  Moreover, for $A \in \mathcal{B}(\mathbb{R})$ we have the occupation density formula
  \begin{equation*}
    \int_A L_t(u) \dd u = \frac{1}{2} \int_0^t \1_A(S(s)) \dd \langle S \rangle(s), \quad t \in [0,\infty).
  \end{equation*}
  In other words, for all $t\ge0$ the occupation measure of $S$ on $[0,t]$ is absolutely continuous with respect to the Lebesgue measure, with density $2 L_t$.
\end{theorem}

\begin{proof}[Sketch of proof]
   Formula~\eqref{eq:discrete local time ito} in combination with the continuity of $f$ and $S$ yields
   \begin{align*}
      f(S( t)) - f(S(0)) &  = \sum_{t_j \in \pi^n} f^\prime (S(t_j)) (S(t_{j+1} \wedge t) - S(t_j \wedge t)) \\
      &\hspace{45pt} + \int_{-\infty}^\infty \Big( \sum_{t_j \in \pi^n} \1_{\llparenthesis S(t_j\wedge t), S(t_{j+1}\wedge t) \rrbracket} (u) \vert S(t_{j+1}\wedge t) - u \vert\Big) f^{\prime\prime}(u) \dd u.
   \end{align*}
   By assumption, the second term on the right hand side converges to $\int_{-\infty}^\infty f^{\prime \prime} (u) L_t(S,u)\dd u$ as $n$ tends to $\infty$, so that also the Riemann sums have to converge.
   
   The occupation density formula follows by approximating $\mathbf{1}_A$ with continuous functions.
\end{proof}

As already observed by Bertoin~\cite{Bertoin1987}, the key point of this extension of F\"ollmer's pathwise stochastic integral is again that it is given by a \emph{continuous} linear operator on $H^1$. Since $L_t(S,\cdot)$ is compactly supported for all $t \ge 0$, the same arguments also work for functions $f$ that are locally in $H^2$, i.e. such that $f|_{(a,b)} \in H^2((a,b),\R)$ for all $-\infty < a < b < \infty$.

As we make stronger assumptions on the local times $L(S)$, it is natural to expect that we can extend Wuermli's generalized It\^o formula to larger spaces of functions.

\begin{definition}\label{def:local2}
  Let $(\pi^n)$ be a sequence of partitions and let $S \in \L_{L^2}(\pi^n)$. We say that $S$ has a \textit{continuous local time} along $(\pi^n)$ if for all $t \in [0,\infty)$ the discrete pathwise local times $L^{\pi^n}_t(S, \cdot)$ converge uniformly to a continuous limit $L_t(S, \cdot)$ as $n \to \infty$ and if $(t,u)\mapsto L_t(S,u)$ is jointly continuous. We write $\L_{c}(\pi^n)$ for the set of all $S$ having a continuous local time along $(\pi^n)$.
\end{definition}

In the following theorem, $\mathrm{BV} = \mathrm{BV}(\R,\R)$ denotes the space of right-continuous bounded variation functions, equipped with the total variation norm.

\begin{theorem}\label{thm:gito}
  Let $(\pi^n)$ be a sequence of partitions and let $S \in \L_{c}(\pi^n)$. Let $f\colon \mathbb{R} \to \mathbb{R}$ be absolutely continuous with right-continuous Radon-Nikodym derivative $f^\prime$ of locally bounded variation. Then we have the generalized change of variable formula 
  \begin{equation*}
    f (S(t)) = f(S(0)) + \int_0^t f^\prime (S(s)) \dd S(s) + \int_{-\infty}^{\infty} L_t(u) \dd f^\prime (u)
  \end{equation*}
  for all $t \in [0,\infty)$, where
  \begin{equation}\label{eq:stochasticIntegral}
    \int_0^t f^\prime  (S(s)) \dd S(s) := \lim_{n \to \infty} \sum_{t_j \in \pi^n} f^\prime  (S(t_j)) (S(t_{j+1} \wedge t) - S(t_j \wedge t)), \quad t \in [0,\infty).
  \end{equation}  
  In particular, the integral $\int_0^\cdot g(S(s)) \dd S(s)$ is defined for all $g$ of locally bounded variation, and for all $T>0$ the map $\mathrm{BV} \ni g \mapsto \int_0^\cdot g(S(s)) \dd S(s) \in C([0,T],\R)$ defines a bounded linear operator.
\end{theorem}

\begin{proof}
  From~\eqref{eq:discrete local time ito} we get
  \begin{equation*}
    f(S( t)) - f(S(0)) =\sum_{t_j \in \pi^n} f^\prime  (S(t_j)) (S(t_{j+1}\wedge t) - S(t_j \wedge t)) 
        + \int_{-\infty}^\infty L^{\pi^n}_t(u) \dd f^\prime (u)
  \end{equation*}
  for all $t \ge 0$. Since $L^{\pi^n}_t$ converges uniformly to $L_t$, our claim immediately follows.
\end{proof}

Observe that $f$ satisfies the assumptions of Theorem~\ref{thm:gito} if and only if it is the difference of two convex functions. For such $f$, Sottinen and Viitasaari \cite{Sottinen2014} prove a generalized change of variable formula for a class of Gaussian processes. They make the very nice observation that for a suitable Gaussian process $X$ one can control the fractional Besov regularity of $f'(X)$, and they use this insight to construct $\int_0^\cdot f'(X_t) \dd X_t$ as a fractional integral. Such a regularity result is somewhat surprising since in general $f'(X)$ is not even l\`adl\`ag, so in particular not of finite $p$-variation for any $p > 0$. But since regularity of $f'(X)$ is shown using probabilistic arguments, the integral of Sottinen and Viitasaari is not directly a pathwise object: the null set outside of which it exists may depend on $f$. Moreover, they can only handle Gaussian processes that are H\"older continuous of order $\alpha > 1/2$, and their approach breaks down when considering processes with non-trivial quadratic variation.
Here we have a completely different focus, since we are interested in pathwise results for paths with non-trivial quadratic variation.

As an immediate consequence of Theorem \ref{thm:gito} we obtain a pathwise version of the classical Tanaka formula.

\begin{corollary}
  Let $(\pi^n)$ be a sequence of partitions and let $S \in \L_{c}(\pi^n)$. The pathwise Tanaka-Meyer formula 
  \begin{equation*}
    L_t(u) = ( S(t) - u )^- - ( S(0) - u )^- + \int_0^t \1_{(-\infty, u)} (S(s)) \dd S(s)
  \end{equation*}
  is valid for all $(t,u) \in [0,\infty) \times \mathbb{R}$, with the notation $(\cdot - u)^- := \max \{ 0, u - \cdot \}$. The analogous formulas for $\1_{[u, \infty)} (\cdot)$ and $\sgn (\cdot - u)$ hold as well.
\end{corollary}

At this point we see a picture emerge: the more regularity the local time has, the larger the space of functions is to which we can extend our pathwise stochastic integral. Indeed, the previous examples are all based on duality between the derivative of the integrand and the occupation measure. In the classical F\"ollmer-It\^o case and for fixed time $T\ge 0$, the occupation measure is just a finite measure on a compact interval $[a,b]$, and certainly the continuous functions belong to the dual space of the finite measures on $[a,b]$. In the Wuermli setting, the occupation measure has a density in $L^2$ and therefore defines a bounded functional on $L^2$. If the local time is continuous, then we can even integrate Radon measures against it.

So if we can quantify the continuity of the local time, then the dual space further increases and we can extend the pathwise It\^o formula to a bigger class of functions. To this end we introduce for a given sequence of partitions $(\pi^n)$ and $p \ge 1$ the set $\mathcal{L}_{c,p}(\pi^n) \subseteq \mathcal{L}_c(\pi^n)$ consisting of those $S \in \mathcal{L}_c(\pi^n)$ for which the discrete local times $(L_t^{\pi^n})$ have uniformly bounded $p$-variation, uniformly in $t \in [0,T]$ for all $T > 0$, i.e. for which
\begin{equation*}
  \sup_{n \in \N} \lVert L^{\pi^n} \rVert_{C_T \V^p} := \sup_{n \in \N} \sup_{t\in [0,T]} \| L_t^{\pi^n}(\cdot)\|_{p\mhyphen\var} < \infty
\end{equation*}
for all $T > 0$, where we write for any $f: \R \rightarrow \R$
\[
   \lVert f \rVert_{p\mhyphen\mathrm{var}}:=\sup\bigg\{ \bigg(\sum_{k=1}^n | f(u_k) - f(u_{k-1})|^{p}\bigg)^{1/p} :\, -\infty < u_0 < \ldots < u_n < \infty ,\, n \in \N\bigg\}.
\]
We also write $\V^p$ for the space of right-continuous functions of finite $p$-variation, equipped with the maximum of the $p$-variation seminorm and the supremum norm.
 \\
For $S \in \L_{c,p}(\pi^n)$ and using the Young integral it is possible to extend the pathwise Tanaka formula to an even larger class of integrands, allowing us to integrate $\int g(S) \dd S$ provided that $g$ has finite $q$-variation for some $q$ with $1/p + 1/q >1$. This is similar in spirit to the Bouleau-Yor~\cite{Bouleau1981} extension of the classical Tanaka formula. Such an extension was previously derived by Feng and Zhao~\cite{Feng2006}, Theorem 2.2. But Feng and Zhao stay in a semimartingale setting, and they interpret the stochastic integral appearing in~\eqref{eq:gito2} as a usual It\^o integral. Here we obtain a pathwise integral, which is given very naturally as a limit of Riemann sums.

Let us briefly recall the main concepts of Young integration. In~\cite{Young1936}, Young showed that if $-\infty < a < b < \infty$, if $f$ and $g$ are two functions on $[a,b]$ of finite $p$- and $q$-variation respectively with $1/p + 1/q >1$, and if $\pi$ is a partition of $[a,b]$, then there exists a universal constant $C(p,q)>0$ such that
\[
   \Big| \sum_{t_j,t_{j+1} \in \pi} f(t_j) (g(t_{j+1}) - g(t_j)) \Big| \le C(p,q) \| f \|_{p\mhyphen\var, [a,b]} \| g\|_{q\mhyphen\var,[a,b]},
\]
where we wrote $\| f \|_{p\mhyphen\var, [a,b]}$ for the $p$-variation of $f$ on $[a,b]$ and similarly for $g$. In particular, if there exists a sequence of partitions $(\pi^n)$ and if the Riemann sums of $f$ against $g$ along $(\pi^n)$ converge to a limit which we denote by $\int_a^b f(s) \dd g(s)$, then
\begin{equation}\label{eq:young bound}
  \Big| \int_a^b f(s) \dd g(s) \Big| \le C(p,q) (|f(a)|+ \| f \|_{p\mhyphen\var, [a,b]}) \| g\|_{q\mhyphen\var,[a,b]}.
\end{equation}
Moreover, Young showed that if $f$ and $g$ have no common points of discontinuity, then the Riemann sums along any sequence of partitions with mesh size going to zero converge to the same limit $\int_0^t f(s) \dd g(s)$, independently of the specific sequence of partitions.

We therefore easily obtain the following theorem.

\begin{theorem}[see also~\cite{Feng2006}, Theorem~2.2]\label{thm:gito2}
  Let $p,q\ge 1$ be such that $ \frac{1}{p} + \frac{1}{q}> 1$. Let $(\pi^n)$ be a sequence of partitions and let $S \in \mathcal{L}_{c,p}(\pi^n)$. Let $f \colon \mathbb{R}\to \mathbb{R}$ be absolutely continuous with right-continuous Radon-Nikodym derivative $f^\prime $ of locally finite $q$-variation. Then for all $t \in [0, \infty)$ the generalized change of variable formula 
  \begin{equation}\label{eq:gito2}
    f(S(t)) =f(S(0)) + \int_0^t f^\prime (S(s))\dd S(s) + \int_{-\infty}^\infty L_t(u) \dd f^\prime (u) 
  \end{equation}
  holds, where $ \dd f^\prime (u)$ denotes Young integration and where
  \[
      \int_0^t f^\prime (S(s))\dd S(s) := \lim_{n \to \infty} \sum_{t_j \in \pi^n} f^\prime  (S(t_j)) (S(t_{j+1} \wedge t) - S(t_j \wedge t)), \quad t \in [0,\infty).
  \]
  In particular, the integral $\int_0^\cdot g(S(s)) \dd S(s)$ is defined for all right-continuous $g$ of locally finite $q$-variation, and for all $T > 0$ the map $\V^q \ni g \mapsto \int_0^\cdot g(S(s)) \dd S(s) \in C([0,T],\R)$ defines a bounded linear operator.
\end{theorem}

\begin{proof}
  Observe that for each $n\in \mathbb{N}$, the discrete local time $L^{\pi^n}_t$ is piecewise smooth and of bounded variation. Therefore, formula~\eqref{eq:discrete local time ito} holds for $L^{\pi^n}_t$ and $f^\prime $, and the integral on the right hand side of~\eqref{eq:discrete local time ito} is given as the limit of Riemann sums along an arbitrary sequence of partitions with mesh size going to zero -- provided that every element of the sequence contains all jump points of $L^{\pi^n}_t$. 
  Therefore, the integral must satisfy the bound~\eqref{eq:young bound}. Since the $p$-variation of $(L^{\pi^n}_t)$ is uniformly bounded, and the sequence converges uniformly to $L_t$, it is easy to see that it must converge in $p'$-variation for all $p'>p$. Choosing such a $p'$ with $1/q + 1/p'>1$ and combining the linearity of the Young integral with the bound~\eqref{eq:young bound}, the result follows.
\end{proof}

\begin{remark}\label{rmk:feng}
  Theorem 2.2 in \cite{Feng2006} states~\eqref{eq:gito2} under the slightly weaker assumption that $f \colon \mathbb{R}\to \mathbb{R}$ is left-continuous and locally bounded with left-continuous and locally bounded left derivative $\DD^-f$ of finite $q$-variation. But absolute continuity of $f$ is clearly necessary: Consider the path $S(t)\equiv t$ for $t \in [0,\infty)$, for which $\langle S \rangle \equiv 0$ and thus $L \equiv 0$. In this case equation \eqref{eq:gito2} would read as
  \begin{equation*}
    f(t)= f(0) + \int_0^t \DD^-f(u)\dd u, \quad t \in [0,\infty),
  \end{equation*}
  a contradiction if $f$ is not absolutely continuous.
\end{remark}

In the following, we will show that any typical price path which might model an asset price trajectory must be in $\mathcal{L}_{c,p}(\pi^n)$ if $(\pi^n)$ denotes the dyadic Lebesgue partition generated by $S$.

\section{Local times for model free finance}\label{sec:localtime}

\subsection{Super-hedging and outer measure}

In a recent series of papers \cite{Vovk2011,Vovk2011a,Vovk2012}, Vovk introduced a hedging based, model free approach to mathematical finance. Roughly speaking, Vovk considers the set of real-valued continuous functions as price paths and introduces an outer measure on this set which is given by the cheapest super-hedging price. A property (P) is said to hold for ``typical price paths'' if it is possible to make an arbitrarily large profit by investing in the paths where (P) is violated. We will see that in Vovk's framework it is possible to construct continuous local times for typical price paths, which gives an axiomatic justification for the use of our pathwise generalized It\^o formulas from Section~\ref{sec:tanaka} in model free finance.

More precisely, we consider the (sample) space $\Omega = C([0,\infty), \mathbb{R})$ of all continuous functions $\omega \colon [0,\infty) \to \mathbb{R}$.
The coordinate process on $\Omega$ is denoted by $S_t(\omega):=\omega(t)$. For $t \in [0,\infty)$ we define $\mathcal{F}_t := \sigma(S_s: s \le t)$ and we set $\mathcal{F}:= \bigvee_{t\geq 0} \mathcal{F}_t$. Stopping times $\tau$ and the associated $\sigma$-algebras $\F_\tau$ are defined as usual.

A process $H \colon \Omega \times [0,\infty) \rightarrow \R$ is called a \emph{simple strategy} if there exist stopping times $0 = \tau_0(\omega) < \tau_1(\omega) < \dots$ 
such that for every $\omega \in \Omega$ and every $T\in (0,\infty)$ we have $\tau_n(\omega) \leq T$ for only finitely many $n$, and $\F_{\tau_n}$-measurable bounded functions 
$F_n\colon \Omega \rightarrow \R$ such that $H_t(\omega) = \sum_{n\ge 0} F_n(\omega) \1_{(\tau_n(\omega),\tau_{n+1}(\omega)]}(t)$. In that case the integral
\begin{align*}
  (H \cdot S)_t(\omega) = \sum_{n=0}^\infty F_n(\omega) [S_{\tau_{n+1}(\omega) \wedge t} - S_{\tau_n(\omega) \wedge t}]
\end{align*}
is well defined for every $\omega \in \Omega$ and every $t \in [0,\infty)$. 

For $\lambda > 0$ a simple strategy $H$ is called $\lambda$-\textit{admissible} if $(H\cdot S)_t(\omega) \ge - \lambda$ for all $t \in [0, \infty)$ and all $\omega \in \Omega$. The set of $\lambda$-admissible simple strategies is denoted by $\mathcal{H}_\lambda$.

\begin{definition}
  The \emph{outer measure} $\overline{P}$ of $A \subseteq \Omega$ is defined as the cheapest superhedging price for $\1_A$, that is
  \begin{align*}
    \overline{P}(A) &:= \inf\Big\{\lambda > 0: \exists(H^n)_{n\in \N} \subseteq \mathcal{H}_\lambda \text{s.t.}\liminf_{t\to \infty} \liminf_{n\rightarrow\infty} (\lambda + (H^n\cdot S)_t (\omega)) \ge \1_A(\omega) \forall \omega \in \Omega\Big\}.
  \end{align*}    
  A set of paths $A \subseteq \Omega$ is called a \textit{null set} if it has outer measure zero. A property (P) holds for \emph{typical price paths} if the set $A$ where (P) is violated is a null set.
\end{definition}

  Of course, it would be more natural to minimize over simple trading strategies rather than over the limit inferior along sequences of simple strategies. But then $\overline{P}$ would not be countably subadditive, and this would make it very difficult to work with. Let us just remark that in the classical definition of superhedging prices in semimartingale models we work with general admissible strategies, and the It\^o integral against a general strategy is given as limit of integrals against simple strategies. So in that sense our definition is analogous to the classical one (apart from the fact that we do not require convergence and consider the $\liminf$ instead).

For us, the most important property of $\overline{P}$ is the following arbitrage interpretation for null sets.

\begin{lemma}[Lemma~2.4 of~\cite{Perkowski2013}]\label{lem:na1}
  A set $A \subseteq \Omega$ is a null set if and only if there exists a sequence of 1-admissible simple strategies $(H^n)_{n\in \mathbb{N}} \subseteq \mathcal{H}_1$, such that
  \begin{align*}
    \liminf_{t\to \infty} \liminf_{n \rightarrow \infty} (1 + (H^n \cdot S)_t (\omega)) \ge \infty \cdot \1_A(\omega),
  \end{align*}
  where we set $\infty \cdot 0 = 0$ and  $\infty \cdot 1 = \infty$.
\end{lemma}

In other words, a null set is essentially a model free arbitrage opportunity of the first kind, and to only work with typical price paths is analogous to only considering models which satisfy (NA1) (no arbitrage opportunities of the first kind). The notion (NA1) has raised a lot of interest in recent years since it is the minimal condition which has to be satisfied by any reasonable asset price model; see for example~\cite{Karatzas2007, Ruf2013, Imkeller2015}. If $\P$ is a probability measure on $(\Omega, \F)$, we say that $S$ satisfies \textit{(NA1)} under $\P$ if the set $\mathcal{W}^{\infty}_1 := \{ 1 + \int_0^\infty H_s \dd S_s\, : \, H \in \mathcal{H}_1 \}$ is bounded in probability, that is if $\lim_{n \to \infty} \sup_{X\in \mathcal{W}^{\infty}_1} \mathbb{P}( X \geq n)=0$. In the continuous setting this is equivalent to $S$ being a semimartingale of the form $S = M + \int_0^\cdot \alpha_s \dd \langle M \rangle_s$, where $M$ is a local martingale and $\int_0^\infty \alpha^2_s \dd \langle M \rangle_s < \infty$.

In the next proposition we collect further properties of $\overline{P}$. For proofs (in finite time) see~\cite{Perkowski2013}.

\begin{proposition}\label{prop:prop}
  \begin{enumerate} 
    \item $\overline{P}$ is an outer measure with $\overline{P}(\Omega)=1$, i.e. $\overline{P}$ is nondecreasing, countably subadditive, and $\overline{P}(\emptyset) = 0$.
    \item Let $\mathbb{P}$ be a probability measure on $(\Omega, \F)$ such that the coordinate process $S$ is a $\mathbb{P}$-local martingale, and let $A \in \F$. Then $\mathbb{P}(A) \le \overline{P}(A)$.
    \item Let $A \in \F$ be a null set, and let $\mathbb{P}$ be a probability measure on $(\Omega, \F)$ such that the coordinate process $S$ satisfies (NA1) under $\P$. Then $\mathbb{P}(A) = 0$.
  \end{enumerate}
\end{proposition}

The last statement says that every property which is satisfied by typical price paths holds quasi-surely for all probability measures which might be of interest in mathematical finance.\medskip

Lemma~\ref{lem:na1} and Proposition~\ref{prop:prop} are originally due to Vovk, but here and in~\cite{Perkowski2013} we consider a small modification of Vovk's outer measure, which in our opinion has a slightly more natural financial interpretation and with which it is easier to work.

\subsection{Existence of local times for typical price paths}

This subsection is devoted to the presentation and the proof of our main result (Theorem~\ref{thm:int}): every typical price path has a local time which satisfies all the requirements needed to apply our most general It\^o-Tanaka formula, Theorem~\ref{thm:gito2}.\\
For this purpose recall that for every partition $\pi(\omega)=\{0=t_0(\omega) < t_1(\omega)< \ldots < t_{K(\omega)}(\omega) < t_{(K+1)(\omega)}(\omega) =\infty \}$ of $[0,\infty)$ a discrete version of the local time is given by 
\begin{equation*}
  L^{\pi}_t( S, u )(\omega) = \sum_{j=0}^{K(\omega)} \1_{\llparenthesis S_{t_j\wedge t}(\omega), S_{t_{j+1}\wedge t}(\omega) \rrbracket} (u) \vert S_{t_{j+1}\wedge t}(\omega) - u \vert, \quad (t,u) \in [0,\infty)\times \mathbb{R}.
\end{equation*}
From~\eqref{eq:discrete local time ito} we get the following discrete version of Tanaka's formula, which can also be obtained by direct computation:
\begin{equation}\label{eq:discrete tanaka}
  L^{\pi}_t (S, u)(\omega)  = (S_t(\omega) - u )^- - (S_0(\omega) - u)^- + \sum_{j = 0}^{K(\omega)} \1_{(-\infty, u )} (S_{t_j}(\omega))[S_{t_{j+1}\wedge t}(\omega)-S_{t_j\wedge t}(\omega)]
\end{equation}
for all $(t,u) \in [0,\infty)\times \mathbb{R}$ and $\omega \in \Omega$. Taking a sequence of partitions with mesh size converging to zero, we see that at least formally the construction of the stochastic integral $\int_0^\cdot \1_{(-\infty,u)}(S_s) \dd S_s(\omega)$ is equivalent to the construction of the local time $L(S,u)(\omega)$.

In the following we will work with a very natural sequence of partitions, namely the dyadic Lebesgue partitions generated by $S$:
For each $n \in \mathbb{N}$ denote $\mathbb{D}^n:= \{ k 2^{-n} \, : \, k \in \mathbb{Z} \}$ and define the sequence of stopping times
\begin{equation}\label{eq:stop}
  \tau^n_0(\omega) := 0, \quad \tau^n_{k+1}(\omega):= \inf \{ t \geq \tau^n_k(\omega) \,: \, S_t(\omega) \in \mathbb{D}^n \setminus S_{\tau^n_k(\omega)}(\omega) \}, \quad  k \in \mathbb{N}.
\end{equation}
We set $\pi^n(\omega):=\{0= \tau^n_0 (\omega)<  \tau^n_1(\omega) < \dots\}$. Note that the functions $\tau^n_k(\omega)$ are stopping times and that $(\pi^n(\omega))$ is increasing, i.e. it holds $\pi^n(\omega)\subset \pi^{n+1}(\omega)$ for all $n \in \mathbb{N}$. From now on we will mostly omit the $\omega$ and just write $\pi^n$ and $\tau^n_k$ instead of $\pi^n(\omega)$ and $\tau^n_k(\omega)$, respectively. \medskip

A key ingredient for our construction of the local time is the following analysis of the number of interval crossings. Let $U_t(\omega,a,b)$ be the number of upcrossings of the closed interval $[a,b] \subseteq \mathbb{R}$ by $S(\omega)$ during the time interval $[0,t]$, where an upcrossing is a pair $(u,v)\in [0,t]^2$ with $u<v$ such that $S_u(\omega)=a$, $S_v(\omega)=b$ and $S_w(\omega) \in (a,b)$ for all $w \in (u,v)$. Downcrossings are defined analogously and we write $D_t(\omega,a,b)$ for the number of downcrossings by  $\omega \in \Omega$ during the time interval $[0,t]$.

\begin{lemma}\label{lem:upcrossings}
  For typical price paths $\omega \in \Omega$, there exists $C(\omega)\colon (0,\infty) \to (0,\infty)$ such that 
  \begin{equation*}
    \max_{k \in \mathbb{Z}}\, \big( U^n_T(\omega,k2^{-n}) + D^n_T(\omega,k2^{-n}) \big)\leq C_T(\omega) n^2 2^n
  \end{equation*}
  for all $n \in \mathbb{N}$, $T>0$, where $U^n_T(\omega,u):=U_T(\omega,u,u + 2^{-n})$ for $u \in \R$, and similarly for the number of downcrossings.
\end{lemma}

\begin{proof}  
  Let $K,T>0$. Without loss of generality we may restrict our considerations to the set $A_{K} :=\{\omega \in \Omega \, :\, \sup_{t \in [0,T]} \vert S_{t} (\omega)\vert < K\}$. Let $k \in (-2^n K, 2^n K)$ and write $u = k2^{-n}$. The following strategy will make a large profit if $U^{n}_{T}(u):=U^n_T(\omega,u)$ is large: start with wealth $1$, when $S$ first hits $u$ buy $1/(2K)$ numbers of shares. When $S$ hits $-K$ sell and stop trading. Otherwise, when $S$ hits $u+2^{-n}$ sell. This gives us wealth $1+2^{-n}/(2K)$ on the set $\{ U^{n}_{T}(u)\geq 1 \} \cap A_{K}$. Now we repeat this strategy: next time we hit $u$, we buy our current wealth times $1/(2K)$ shares of $S$, and sell when $S$ hits $u+2^{-n}$ or $-K$. After $n^{2} 2^{n}$ upcrossings of $[u, u+2^{-n}]$, stop trading. On the set $\{ U^{n}_{T}(u) \geq n^{2} 2^{n} \} \cap A_{K}$ we then have a wealth of
  \begin{equation*}
    \Big( 1+ \frac{2^{-n}}{2K} \Big)^{n^{2} 2^{n}} \ge \exp \Big(\frac{1}{4K} n^{2} \Big)
  \end{equation*}
  for all $n$ that are large enough. Therefore
  \begin{equation*}
    \bar{P} \big( \{ U^{n}_{T} ( u ) \geq n^{2} 2^{n} \} \cap A_{K} \big) \leq \exp \Big( -  \frac{n^{2}}{4K} \Big)
  \end{equation*}
  for all large $n$. Summing over all dyadic points $u= k 2^{-n}$ in $( -K,K )$, we obtain
  \[
    \overline{P} \Big( \Big\{ \max_{k \in \mathbb{Z}} U^{n}_{T} (k 2^{-n}) \geq n^{2} 2^{n} \Big\} \cap A_{K} \Big)
    \leq K2^{n+1} \exp \Big(-\frac{n^{2}}{4K} \Big) = K  \exp \Big( -  \frac{n^{2}}{8K} + (n+1) \log (2) \Big)
  \]
  for all large $n$. Since this is summable in $n$, the claimed bound for the upcrossings follows for all typical price paths. To bound the downcrossings, it suffices to note that up- and downcrossings of a given interval differ by at most 1.
\end{proof}

The following construction is partly inspired by~\cite{Morters2010}, Chapter~6.2.

\begin{theorem}\label{thm:int}
  Let $T >0$, $\alpha \in (0,1/2)$ and $(\pi^n)$ as defined in \eqref{eq:stop}. For typical price paths $\omega \in \Omega$, the discrete local time $L^{\pi^n}(S,\cdot)$ converges uniformly in $(t,u) \in [0,T] \times \R$ to a limit $L(S,\cdot) \in C([0,T], C^\alpha (\R))$, and there exists $C=C(\omega)>0$ such that 
  \begin{equation}\label{eq:indicator}
    \sup_n \Big\{ 2^{n \alpha} \vert\vert L^{\pi^{n}}(S,\cdot) -L(S,\cdot) \vert\vert_{L^{\infty} ( [ 0,T ] \times \mathbb{R} )} \Big\} \leq C.
  \end{equation}
  Moreover, for all $p>2$ we have $\sup_{n \in \mathbb{N}}\vert\vert L^{\pi^{n}} \vert\vert_{C_{T} \mathcal{V}^{p}} < \infty$ for typical price paths.
\end{theorem}
 
\begin{proof}
  By the identity~\eqref{eq:discrete tanaka} it suffices to prove the corresponding statements with the stochastic integrals $\int_0^t 1_{(-\infty,u)}(S_s)\dd S_s$ replacing $L_t(S,u)$. Using Lemma~\ref{lem:upcrossings}, we may fix $K>0$ and restrict our attention to the set
  \begin{equation*}
    A_{K} := \bigg\{\omega \in \Omega : \sup_{t \in [ 0,T ]} \vert S_{t} (\omega)\vert < K \text{ and }  \max_{k \in \mathbb{Z}}\, \big( U^n_T(\omega,k2^{-n}) + D_T(\omega,k2^{-n}) \big)\leq K n^2 2^n \,\, \forall \, n \bigg\}.
  \end{equation*}
  Let $u \in (-K,K)$. For every $n \in \mathbb{N}$ we approximate $\1_{(-\infty, u)}(S)$ by the process
  \begin{equation*}
    F^n_t(u) := \sum_{k=0}^\infty \1_{(-\infty, u)}(S_{\tau^n_k}) \1_{[\tau^n_k, \tau^n_{k+1})}(t), \quad t \ge 0.
  \end{equation*}
  Then we write for the corresponding integral process 
  \begin{equation*}
    I^{\pi^{n}}_{t}(u):=\sum_{k=0}^\infty \1_{(-\infty, u )}(S_{\tau_k^n}(\omega))[S_{\tau_{k+1}^n \wedge t}(\omega)-S_{\tau_k^n \wedge t}(\omega)], \quad t \ge 0,
  \end{equation*}
  and since $(\pi^n)$ is increasing, we get
  \begin{equation*}
    I^{\pi^{n}}_{t}(u) -I^{\pi^{n-1}}_{t}(u) = \sum_{k=0}^{\infty} [F_{\tau_k^n}^n(u)-F_{\tau_k^n}^{n-1}(u)][S_{\tau^n_{k+1}\wedge t}-S_{\tau^n_k \wedge t}].
  \end{equation*}
  By the construction of our stopping times $(\tau^n_k)$, we have
  \begin{equation*}
    \sup_{t\ge 0} \big\vert [F^n_{\tau_k^n}(u) - F^{n-1}_{\tau_k^n} (u)][S_{\tau^n_{k+1}\wedge t}(\omega)-S_{\tau^n_k \wedge t}(\omega)] \big\vert \leq 2^{-n+2}.
  \end{equation*}
  Hence, the pathwise Hoeffding inequality, Theorem~3 in~\cite{Vovk2012} or Lemma~A.1 in~\cite{Perkowski2013}, implies for every $\lambda \in \R$ the existence of a 1-admissible simple strategy $H^\lambda \in \mathcal{H}_1$, such that
  \begin{equation*}\label{e:karandikar integral pr1}
    1 + (H^\lambda \cdot S)_t(\omega) \ge \exp\bigg( \lambda (I^{\pi^{n}}_{t}(u) -I^{\pi^{n-1}}_{t}(u) ) - \frac{\lambda^2}{2} N^n_t(u,\omega) 2^{-2n+4}\bigg) =: \mathcal{E}^{\lambda,n}_t(\omega)
  \end{equation*}
  for all $t\in [0,T]$ and all $\omega \in \Omega$, where $N^n_t(u) := N^n_t(u,\omega)$ denotes the number of stopping times $\tau^n_k \leq t$ with $F^n_{\tau_k^n}(u) - F^{n-1}_{\tau_k^n} (u)\neq 0$. Now observe that $F^{n}_{t}$ and $F^{n-1}_{t}$ are constant on dyadic intervals of length $2^{-n}$, which means that we may suppose without loss of generality that $u= k 2^{-n}$ is a dyadic number. But we can estimate $N^{n}_{T} (k 2^{-n} )$ by the number of upcrossings of the interval $[(k -1) 2^{-n} , k 2^{-n} ]$ plus the number of the downcrossings of the interval $[k 2^{-n},(k +1) 2^{-n}]$, which means that on $A_K$ we have $N^n_T(u) \le 2 K 2^n n^2$.     
  So considering $(H^{\lambda}+H^{-\lambda})/2$ for $\lambda >0$, we get
  \begin{equation*}
    \overline{P} \bigg( \bigg\{ \sup_{t \in [ 0,T ]} | I^{\pi^{n}}_{t}(u) -I^{\pi^{n-1}}_{t}(u) | \geq 2^{-n \alpha}\bigg \} \cap A_K \bigg)
       \leq 2 \exp ( - \lambda 2^{-n \alpha} + \lambda^{2} K 2^{-n+4} n^{2})
  \end{equation*}
  for all $\lambda, \alpha>0$. Choose now $\lambda =2^{n/2}$ and $\alpha \in (0,1/2)$. Then we get the estimate
  \begin{equation*}
     \overline{P} \bigg( \bigg\{ \sup_{t \in [0,T]} \vert I^{\pi^{n}}_{t} (u) -I^{\pi^{n-1}}_{t} (u) \vert \geq 2^{-n \alpha} \bigg\} \cap A_{K}\bigg)
       \leq 2 \exp ( -2^{n ( 1/2- \alpha )} +16 K n^{2} ).
  \end{equation*}
  Moreover, noting that for all $t>0$ the maps $u \mapsto I^{\pi^{n}}_t(u)$ and $u \mapsto I^{\pi^{n-1}}_t(u)$ are constant on dyadic intervals of length $2^{-n}$ and that there are $2K2^{n}$ such intervals in $[-K,K]$, we can simply estimate
  \begin{align*}
    \overline{P}\bigg(\bigg \{ \sup_{(t,u) \in [0,T] \times \mathbb{R}} \vert I^{\pi^{n}}_{t} (u) -  I^{\pi^{n-1}}_{t} (u) & \vert \geq 2^{-n\alpha}\bigg\} \cap A_K \bigg) \\
       &\leq 2K2^{n} \times 2 \exp ( -2^{n (1/2-\alpha)} +16 K n^{2})\\
       &= \exp ( -2^{n (1/2-\alpha)} +16 Kn^{2} +(n+2)\log 2+ \log K).
  \end{align*}
  Obviously, this is summable in $n$ and thus the proof of the uniform convergence and of the speed of convergence is complete.
  
  It remains to prove the uniform bound on the $p$-variation norm of $I^{\pi^{n}}$ and the H\"older continuity of the limit. Let $p>2$ and write $\alpha =1/p$, so that $\alpha \in (0,1/2)$. First let $u = k 2^{-n} \in (-K,K)$ and write $v= (k+1) 2^{-n}$. Then
  \begin{equation*}
    I^{\pi^{n}}_{t} (v) -I^{\pi^{n}}_{t}(u) = \sum_{k=0}^{\infty} (F^{n}_{\tau^{n}_{k}}(v) -F^{n}_{\tau^{n}_{k}} ( u ) ) (S_{\tau^{n}_{k}\wedge t} -S_{\tau^{n}_{k-1}\wedge t} ) ,
  \end{equation*}
  and similarly as before we have $\sup_{t \geq 0}\vert ( F^{n}_{\tau^{n}_{k}}(v)-F^{n}_{\tau^{n}_{k}} ( u ) ) ( S_{\tau^{n}_{k}\wedge t} -S_{\tau^{n}_{k-1}\wedge t} )\vert \leq 2^{-n+1}$. On $A_K$, the number of stopping times $(\tau^n_k)_k$ with $F^{n}_{\tau^{n}_{k}}(u) \neq F^{n}_{\tau^{n}_{k}}(v)$ is bounded from above by $2K 2^n n^2 +1$, and therefore we can estimate as before
  \begin{equation*}
    \overline{P}\bigg(\bigg\{\sup_{t \in [0,T]} \sup_{u,v \in \R : \vert u-v \vert \leq 2^{-n}} \vert I^{\pi^{n}}_{t}(v) - I^{\pi^{n}}_{t}(u) \vert \geq 2^{-n \alpha}\bigg \} \cap A_K \bigg)\leq \exp (-2^{n (1/2- \alpha)} +C  n^{2}),
  \end{equation*}
  for some appropriate constant $C=C(K)>0$.
  
  We conclude that for typical price paths $\omega \in \Omega$ there exists $C = C(\omega) >0$ such that
  \begin{equation*}
    \sup_{t \in [0,T]} \sup_{\vert u-v \vert \leq 2^{-n}} \vert I^{\pi^{n}}_{t} (v) -I^{\pi^{n}}_{t} (u) \vert + \sup_{t \in [0,T]} \sup_{u \in \R} \vert I^{\pi^{n}}_{t} (u)-I^{\pi^{n-1}}_{t} (u) \vert\leq C 2^{-n \alpha}
  \end{equation*}
  for all $n \in \N$. Let now $n \in \mathbb{N}$ and let $u,v \in \R$ with $1 \geq \vert u-v\vert \geq 2^{-n}$. Let $m \leq n$ be such that $2^{-m-1} < \vert u-v \vert \leq 2^{-m}$. Then
  \begin{align*} 
    \vert\vert I^{\pi^{n}} (v) - I^{\pi^{n}} (u)& \vert\vert_{\infty} \\
    &\leq \vert\vert I^{\pi^{n}} (v) -I^{\pi^{m}} (v) \vert\vert_{\infty} +\vert\vert I^{\pi^{m}}(v) -I^{\pi^{m}}(u) \vert\vert_{\infty} + \vert\vert I^{\pi^{m}}(u) -I^{\pi^{n}}(u) \vert\vert_{\infty} \\
    &\leq C \left(\sum_{k=m+1}^{n} 2^{-k \alpha} +2^{-m \alpha} + \sum_{k=m+1}^{n} 2^{-k \alpha} \right)
    \leq C 2^{-m \alpha} \leq C \vert v-u \vert^{\alpha},
  \end{align*}
  possibly adapting the value of $C>0$ in every step. Since $I^{\pi^{n}}_{t}$ is constant on dyadic intervals of length $2^{-n}$, this proves that $\sup_{t \in [ 0,T ]}\vert\vert I^{\pi^{n}}_{t} \vert\vert_{p\mhyphen\var} \le C$. The $\alpha$-H\"older continuity of the limit is shown in the same way.
\end{proof}

We reduced the problem of constructing $L$ to the problem of constructing certain integrals. In~\cite{Perkowski2013}, Corollary~3.6, we gave a general pathwise construction of stochastic integrals. But this result does not apply here, because in general $\1_{(-\infty, u)}(S)$ is not c\`adl\`ag.

\begin{remark}
  Theorem~\ref{thm:int} gives a simple, model free proof that local times exist and have nice properties. Let us stress again that by Proposition~\ref{prop:prop}, all the statements of Theorem~\ref{thm:int} hold quasi-surely for all probability measures on $(\Omega, \F)$ under which $S$ satisfies (NA1).
   
  Below, we sketch an alternative proof based on Vovk's pathwise Dambis Dubins-Schwarz theorem. While we are interested in a statement for typical price paths, which a priori is stronger than a quasi-sure result for all measures satisfying (NA1), the quasi-sure statement may also be obtained by observing that every process satisfying (NA1) admits a dominating local martingale measure, see~\cite{Ruf2013, Imkeller2015}. Under the local martingale measure we can then perform a time change to turn the coordinate process into a Brownian motion, and then we can invoke standard results for Brownian motion for which all statements of Theorem~\ref{thm:int} except one are well known: The only result we could not find in the literature is the uniform boundedness in $p$-variation of the discrete local times.
\end{remark}

\begin{remark}
  Note that for $u=k 2^{-n}$ with $k \in \Z$ we have $L^{\pi^n}_t(u)=2^{-n} D_t(u-2^{-n}, u) + \varepsilon(n,t,u)$ for some $\varepsilon(n,t,u) \in [0,2^{-n}]$. Therefore, our proof also shows that the renormalized downcrossings converge uniformly to the local time, with speed at least $2^{-n\alpha}$ for $\alpha<1/2$. For the Brownian motion this is well known, see~\cite{Chacon1981}; see also~\cite{Khoshnevisan1994} for the exact speed of convergence. In the Brownian case, we actually know more: Outside of one fixed null set we have
  \[
    \lim_{\varepsilon \to 0} \sup_{x \in \R} \sup_{t \in [0,T]} | \varepsilon^{-1} D_t(x, x+\varepsilon) - L_t(x)| = 0
  \]
  for all $T>0$. It should be possible to recover this result also in our setting. It follows from pathwise estimates once we prove Theorem~\ref{thm:int} for a sequence of partitions $(\widetilde \pi^n)$ of the following type: Let $(c_n)$ be a sequence of strictly positive numbers converging to 0, such that $c_{n+1}/c_n$ converges to 1. Define $\widetilde {\mathbb D}^n := \{ k c_n \, : \, k \in \Z\}$. Now define $\widetilde \pi^n$ as $\pi^n$, replacing $\mathbb D^n$ by $\widetilde {\mathbb D}^n$. The only problem is that then we cannot expect the sequence $(\widetilde \pi^n)$ to be increasing, and this would complicate the presentation, which is why we prefer to work with the dyadic Lebesgue partition.
\end{remark}

Finally, we want to indicate that Theorem \ref{thm:int} could also be partially proven by relying on the pathwise Dambis Dubins-Schwarz type theorem of Vovk~\cite{Vovk2012}, which allows to transfer properties of the one-dimensional Wiener process to typical price paths.\\
As mentioned above, Vovk's outer measure $\overline{Q}$ is defined slightly differently than $\overline{P}$ but all results which hold true outside of a $\overline{Q}$-null set are also true outside of a $\overline{P}$-null set; see Section 2.4 of~\cite{Perkowski2013}. To understand Vovk's pathwise Dambis Dubins-Schwarz theorem, we need to recall the definition of time-super\-invariant sets.
\begin{definition}
  A continuous non-decreasing function $f \colon [0, \infty) \to [0,\infty)$ satisfying $f(0)=0$ is said to be a \textit{time change}. A subset $A \subseteq \Omega$ is called 
  \textit{time-superinvariant} if for each $\omega \in \Omega$ and each time change $f$ it is true that $\omega\circ f \in A$ implies $\omega \in A$.
\end{definition}
Roughly speaking, Vovk proved in Theorem 3.1 of~\cite{Vovk2012} that the Wiener measure of a time-super\-invariant set equals the outer measure $\overline{Q}$ of this set. It turns out that the sets
\begin{align*}
  A_c &:= \{ \omega \in \Omega \, : \,  S(\omega) \in \L_c \}\quad \text{and} \\
  A_{c,p} &:= \{ \omega \in A_c \, : \, u \mapsto L_t(S,u)(\omega) \text{ has finite $p$-variation for all } t \in [0,\infty) \}
\end{align*}
are time-super\-invariant. Based on this, one can rely on classical results for the Wiener process (see~\cite{Karatzas1988}, Theorem~3.6.11 or~\cite{Morters2010}, Theorem~6.19) to show that typical price paths have an absolutely continuous occupation measure $L_t(S,u)$  with jointly continuous density and that $L_t(S,\cdot)$ has finite $p$-variation which is uniformly bounded in $t \in [0,T]$ for all $T > 0$ and all $p>2$ (see~\cite{Morters2010}, Theorem~6.19).

However, to the best of our knowledge the alternative approach does not give us the uniform boundedness in $p$-variation of the approximating sequence $(L^{\pi^n})$: we were not able to find such a result in the literature on Brownian motion. Without this, we would only be able to prove an abstract version of Theorem~\ref{thm:gito2}, where the pathwise stochastic integral $\int_0^t g(S_s) \dd S_s$ is defined by approximating $g$ with smooth functions for which the F\"ollmer-It\^o formula Theorem~\ref{thm:ito} holds (see~\cite{Feng2006} for similar arguments in a semimartingale context). Since we are interested in the Riemann sum interpretation of the pathwise integral, we need Theorem~\ref{thm:int} to make sure that all requirements of Theorem~\ref{thm:gito2} are satisfied for typical price paths.


\end{document}